\documentclass[11pt,a4paper]{amsart}
\usepackage{amssymb,amsmath}
\usepackage{color}
\usepackage{latexsym}
\usepackage{amsthm,amsfonts,amssymb,mathrsfs}
\usepackage{rotating}
\usepackage[leqno]{amsmath}
\usepackage{xspace}
\usepackage[all]{xy}
\usepackage{longtable}

\headsep=1cm \numberwithin{equation}{section}
\newtheorem{theorem}{Theorem}[section]
\newtheorem{lemma}[theorem]{Lemma}
\newtheorem{proposition}[theorem]{Proposition}
\newtheorem{corollary}[theorem]{Corollary}

\theoremstyle{definition}
\newtheorem{definition}[theorem]{Definition}
\theoremstyle{remark}
\newtheorem{remark}[theorem]{Remark}
\newtheorem{example}[theorem]{Example}
\newtheorem{question}[theorem]{Question}

\newtheorem{acknowledgement}{Acknowledgement}

\newcommand{\pgrade}{\operatorname{p.grade}}

\newcommand{\cgrade}{\operatorname{c.grade}}
\newcommand{\Cgrade}{\operatorname{\check{C}.grade}}
\newcommand{\wAss}{\operatorname{wAss}}
\newcommand{\Spec}{\operatorname{Spec}}

\newcommand{\rad}{\operatorname{rad}}

\newcommand{\Ht}{\operatorname{ht}}

\newcommand{\V}{\operatorname{V}}
\newcommand{\Var}{\operatorname{Var}}

\newcommand{\Supp}{\operatorname{Supp}}

\newcommand{\lo}{\longrightarrow}
\newcommand{\fm}{\frak{m}}
\newcommand{\fp}{\frak{p}}
\newcommand{\fq}{\frak{q}}
\newcommand{\fa}{\frak{a}}
\newcommand{\fb}{\frak{b}}

\begin{document}

\author[Asgharzadeh, Dorreh and Tousi]{Mohsen Asgharzadeh, Mehdi Dorreh and Massoud Tousi}

\title[ Direct limits of Cohen-Macaulay rings]
{Direct limits of Cohen-Macaulay rings}

\address{M. Asgharzadeh, School of Mathematics, Institute for Research in Fundamental
Sciences (IPM), P. O. Box 19395-5746, Tehran, Iran.}
\email{asgharzadeh@ipm.ir}
\address{M. Dorreh,  Department of Mathematics Shahid Beheshti
University, G.C., Tehran, Iran.}
\email{mdorreh@ipm.ir}
\address{M. Tousi, Department of Mathematics Shahid Beheshti
University, G.C., Tehran, Iran, and   School of Mathematics,  Institute for Research in Fundamental
Sciences (IPM), P. O. Box 19395-5746, Tehran, Iran.}

\email{mtousi@ipm.ir} 


\keywords{ $\check{C}$ech cohomology; Cohen-Macaulay rings; desingularization; direct limit; flat extensions; non-Noetherian rings; normal semigroups; pure extensions.\\The first author   was  supported by a grant from IPM, no. 91130407.\\
The third author   was  supported by a grant from IPM, no. 91130211.
}

\begin{abstract}
Let $A$ be a direct limit of a direct system of Cohen-Macaulay rings. In this paper, we describe the Cohen-Macaulay property of $A$.
Our results indicate that $A$ is not necessarily Cohen-Macaulay. We show $A$ is Cohen-Macaulay under various assumptions.
As an application, we study Cohen-Macaulayness of non-affine normal semigroup rings.
\end{abstract}

\maketitle

\section{Introduction}

Let $H \subseteq \mathbb{Z}^n$ be a  \textit{normal} and \textit{affine} semigroup and let $k$ be a  field. A well-known result of Hochster  says that $k[H]$ is Cohen-Macaulay, see \cite[Theorem 1]{H}. Suppose now that $H$ is a non-affine normal  semigroup. Then $k[H]$ is not Noetherian.
Recently, the notion of Cohen-Macaulayness has been generalized to the non-Noetherian situation, see \cite{HM} and \cite{AT}.
 We intend to investigate the Cohen-Macaulayness of $k[H]$ in this context. In view of  Example \ref{ha1}, $k[H]$ is a direct limit of Noetherian Cohen-Macaulay rings when  $H \subseteq \mathbb{Z}^n$ is normal. This motivates us to ask:

\begin{question}
Is Cohen-Macaulayness closed under taking direct limit?
\end{question}

Our aim in this paper is to study the above question and its connection with Cohen-Macaulay properties of normal semigroup rings. Noetherian rings  which are flat direct limit of a family of rings with certain properties
was studied by Marot \cite{M} and Doretti  \cite{D}.
Also, the regular and complete intersection analogues of  Question 1.1 follows by an immediate application of the theory of Andre-Quillen cohomology  when the direct limit is Noetherian, see Remark \ref{rem}.

There are many candidates for definition of non-Noetherian Cohen-Macaulay rings, see \cite[Definition 3.1]{AT}.
To see relations between them, we recommend \cite[3.2 Relations]{AT}.
All of them are the same, if the base ring is Noetherian.
By using Zariski desingularization,
we give a direct system of Noetherian regular rings $\{A_i:i\in I\}$ such that $\varinjlim_{i\in I}A_i$ is not  Cohen-Macaulay in the sense of ideals, see Definition \ref{def1} and Example \ref{coexam}.
In Section 3 we show  that  direct limit of Noetherian   Cohen-Macaulay rings is Cohen-Macaulay in the sense of ideals in many situations, see Proposition \ref{cm} and Corollaries \ref{drnoeth},  \ref{main2} and \ref{cor1} for its application.
For the Gorenstein version of Question 1.1, see Corollary \ref{cor2}.

Considering \cite{HM}, there is a description of the Cohen-Macaulay property for non-Noetherian rings.
 We call it Cohen-Macaulayness in the sense of Hamilton-Marley. In Section 4, we study  the behavior of this property under taking direct limit, see Proposition \ref{purity}.
As an application, we give  Cohen-Macaulayness of infinite dimensional Veronese and locally finitely generated normal semigroup rings in the sense of Hamilton-Marley, see Corollary \ref{semi2} and Example \ref{br2}.
In Example \ref{coexamhm1}, we give a direct system of  Noetherian Cohen-Macaulay rings such that its direct limit
is not Cohen-Macaulay in the sense of Hamilton-Marley. Let $H$ be the normal semigroup $  \{(a,b) \in \mathbb{N}_0^2| 0 \leq b/a < \infty\} \cup \{(0,0)\}$ and let $k$ be a  field.
Note that $H$ is a standard example of a non-affine semigroup. As an application of Example \ref{coexamhm1}, in Theorem \ref{coexamhm2},
we show
\begin{enumerate}
\item[$\mathrm{(i)}$]$k[H]$ is not  Cohen-Macaulay in the sense of ideals.
\item[$\mathrm{(ii)}$]$k[H]$ is not weak bourbaki unmixed.
\item[$\mathrm{(iii)}$]
 $k[H]$ is  Cohen-Macaulay in the sense of Hamilton-Marley.
\end{enumerate}

Throughout this paper,  rings are commutative (not necessarily Noetherian).
Most of the results concerning direct limits in this paper require the direct system to be filtered.
In the sequel all direct systems are assumed to be filtered.

\section{direct limit of Cohen-Macaulay rings: Examples}

In this section we construct some examples  of rings that are  direct limits of Noetherian Cohen-Macaulay rings.
In what follows, we will use all of them. Also, we study their Cohen-Macaulay properties  in Examples \ref{coexam}, \ref{coexamhm1} and Theorem \ref{coexamhm2}.

\begin{definition}
Let $C$ be a submonoid  of  $\mathbb{Z}^n$ for some $n$  and let $k$ be a field.
\begin{enumerate}
\item[(i)]We write $k[C]$ for the vector space $k^{(C)}$, and denote the basis element of $k[C]$ which corresponds to $ c\in C$ by $ X^ c$. This monomial notation is suggested by the fact that $k[C]$ carries a natural multiplication whose table is given by $X^c X^{c^\prime}= X^{c+c^\prime}$.

\item[(ii)] Recall that $C\subseteq \mathbb{Z}^n$ is said to be \emph{normal} if, whenever
$m( c-c^\prime) \in C$ for some positive integer $m$ and
$c,c^\prime \in C$, then $ c-c^\prime \in C$.
\item[(iii)] Let $N$ be a submonoid  of a commutative monoid  $M$. The \emph{integral closure} of $N$ in $M$ is the submonoid  $$ \widehat{N}_M:=\{ x\in M| mx \in N\; for \;some \;positive\; integer\; m\in   \mathbb{N}\}$$ of $M$. One calls $N$ integrally closed in $M$, if $N = \widehat{N}_M$.
           \end{enumerate}
\end{definition}

\begin{lemma}\label{ha}
Let $C$ be a   normal submonoid  of  $\mathbb{Z}^n$. Then there is a direct system $(C_\gamma , f_{\gamma \delta })$ of finitely generated  normal submonoids  of $C$ such that  $C=\varinjlim_{\gamma \in \Gamma}C_\gamma$,  where $f_{\gamma \delta }: C_\gamma \to  C_\delta$ is the inclusion map for $ \gamma,\delta \in \Gamma$ with $ \gamma \leq \delta$.
\end{lemma}

\begin{proof}
Look at the set $\Gamma :=\{X\subseteq C| X \; is \; finite\}$. We direct $\Gamma$ by means of inclusion. Set $C_X:= \mathbb{Z}^{\geq 0} X$ and $C_X ^ \prime := \widehat{(C_X)}_{\mathbb{Z}C_X}$ for all $X \in \Gamma$. One sees easily that $C_X ^ \prime$ is normal. In view of \cite[Proposition 2.22]{BG}, $C_X ^\prime$ is a finitely generated  monoid.   Clearly, $C_X ^ \prime \subseteq C_Y ^ \prime$ when $X \subseteq Y$ and $X,Y \in \Gamma$.  Let $f_{XY}: C_X ^ \prime \longrightarrow C_Y^ \prime$ be the inclusion map.

  In order to show $C=\varinjlim_{X \in \Gamma}C_X ^ \prime$ we prove  that $C_X ^ \prime \subseteq C$ for all $X \in \Gamma$.
To this end, take $X \in \Gamma$ and  $ x\in C_X ^\prime$. Then $x= \alpha -\beta$ for $\alpha , \beta \in C_X$ and there is a positive integer $m$ such that $m(\alpha -\beta) \in C_X \subseteq C$. It derives from the normality of $C$ that $x = \alpha - \beta \in C$, i.e., $C_X ^ \prime \subseteq C$.  This finishes the proof.
 \end{proof}

\begin{example}\label{ha1}
Let $C$ be a   normal submonoid of  $\mathbb{Z}^n$. Then $k[C]$ is a direct limit
 of its Noetherian Cohen-Macaulay subrings.
\end{example}

\begin{proof}
In view of Lemma \ref{ha}, there is a direct system $(C_\gamma , f_{\gamma \delta })$ of finitely generated  normal submonoids of $C$ such that  $C=\varinjlim_{\gamma \in \Gamma}C_\gamma$. By \cite[Theorem 6.3.5]{BH}, $k[C_\gamma]$ is a Noetherian Cohen-Macaulay ring. Thus, $k[C]$ represented by the direct limit of Noetherian Cohen-Macaulay rings $\{k[C_\gamma]\}$.
\end{proof}

\begin{corollary}\label{coexamhm}
Let $k$  be a field. Then $k + xk[x,y]$ is a direct limit
of a chain of Noetherian Cohen-Macaulay rings.
\end{corollary}

\begin{proof}
Set $H:= (\mathbb{N} \times \mathbb{N}_0) \cup \{(0,0)\}$ and  $A:=k + xk[x,y]$. Then
$A=k[H]$.  Note that $$H =
\{(a,b) \in \mathbb{N}_0^2| 0 \leq b/a < \infty\} \cup \{(0,0)\}.$$
Let $(a,b),(c,d)\in H$ and  suppose there is $k \in \mathbb{N}$ such that $k((a,b)-(c,d))\in  H$. Clearly, $$0 \leq (b-d)/(a-c)=k(b-d)/k(a-c)<\infty$$ and so $H$ is normal.
Example \ref{ha1} yields the claim.
\end{proof}

Let $A$ be a ring, $B$ an $A$-algebra; and $I$ an ideal of $B$. $B$ is called $I$-smooth over $A$ if for given an $A$-algebra $C$, an ideal $N$ of $C$ satisfying ${N}^2 = 0$, and an $A$-algebra homomorphism $u: B \to C/N$ such that $u({I}^v) = 0$ for some $v$, then there exists a lifting $g : B \to C$ of $u$ to $C$:
$$
\xymatrix{B \ar[r]^u  \ar[rd]^g & C/N  \\
 A \ar[u]\ar[r] & C \ar[u]\space{}.}
$$
If $I = (0)$ then $B$ is called \textit{smooth} over $A$. Let $S$ be a multiplicative closed subset of $A$. Then, $S^{-1}A$ is $I$-smooth over $A$
for each ideal $I$ of  $S^{-1}A$.

We cite the following
key result of Zariski.

\begin{lemma}\label{zar}
Let $V$ be a valuation domain containing a field  $k$ of zero characteristic. Then $V=\varinjlim_{i\in I}A_i$, where $A_i$ is  essentially of finite type smooth $k$-algebra.
\end{lemma}

\begin{proof}
This is in \cite{Z}, when the fraction field of $V$ is an algebraic function field over $k$. The general case (follows by an inductive
limit) is in \cite[(1.2) Theorem]{P}.
\end{proof}

\begin{example}\label{ex}
Let $(V,\fm)$ be a valuation domain containing a field $k$ of characteristic zero.
Then $V$ is a direct limit of Noetherian regular rings.
\end{example}

\begin{proof}
By Lemma \ref{zar}, $V=\varinjlim_{i\in I}A_i$ where $A_i$ is smooth and essentially of finite type over $k$. We note that smooth $k$-algebras
  are regular, see  \cite[Theorem 28.1 and Page 216, Lemma]{Mat}.  Thus $V$ is a direct limit of Noetherian regular rings.
\end{proof}

The following are easy examples of the notion of locally finitely generated semigroup rings, see Definition \ref{gen}.

\begin{example}
$\mathrm{(i)}$: Look at
the ring
$R[X_1,\ldots]:=\bigcup_{i=1}^{\infty}R[X_1,\ldots,X_i],$ when
$R$ is Noetherian and Cohen-Macaulay. Clearly, $R[X_1,\ldots]$ is a direct limit
of Noetherian Cohen-Macaulay rings.

$\mathrm{(ii)}$:
Let $k$ be a field and look at $A:=k[X_1,\ldots]$. The degree of $f:=X_{i_1}^{j_1}\ldots X_{i_{\ell}}^{j_{\ell}}\in A$ is defined by $d(f):=\sum_{k=1}^{\ell} j_k$. Let $d$ be a positive integer. We call the $k$-algebra  $A^{(d)}$ generated by all monomials of degree $d$, the $d$-th Veronese subring of $A$. Suppose in addition that $A_n^{(d)}$ is the $d$-th Veronese subring of $k[X_1,\ldots,X_n]$. By \cite[Exercise 12.4.8]{BS}, $A_n^{(d)}$ is a Noetherian Cohen-Macaulay ring.  Remark that $A^{(d)}=\bigcup_{n=1}^{\infty} A_n^{(d)}$. Thus $A^{(d)}$ is  a direct limit
of Noetherian Cohen-Macaulay rings.
\end{example}

\section{ Cohen-Macaulayness in the sense of ideals}

The key result of this Section is Proposition \ref{cm}. We present several
applications of it, see Corollaries \ref{drnoeth}, \ref{main2} and \ref{cor2}.

\begin{definition}\label{pol}
Let $R$ be a ring, $M$ an $R$-module and  $\underline{x} = x_{1}, \ldots,
x_\ell$ be a system of elements of $R$. The sequence $\underline{x}$ is called  a weak regular sequence on $M$ if $x_i$ is a nonzero-divisor on $M/(x_1,\ldots, x_{i-1})M$. The classical grade  of an ideal $\fa$ on $M$, denoted by $\cgrade_R(\fa,M)$, is defined to the supremum of the lengths of all weak regular sequences on $M$ contained in $\fa$.
We use non-Noetherian grade of $\fa$ on $M$, referred  to as the \textit{polynomial grade} and  defined by $$\pgrade_R(\fa,M):=\underset{m\rightarrow\infty}{\lim}\cgrade_{R[t_1, \ldots,t_m]}(\fa R[t_1, \ldots,t_m],R[t_1,,\ldots,t_m]\otimes_R M).$$
\end{definition}

It may be worth to recall that the classical grade coincides with the polynomial grade if the ring and the module both are Noetherian.

\begin{definition}\label{def1}
Let $A$ be a ring with an ideal $I$ and $M$ an $A$-module.
Denote the minimal number of
elements of $A$ that need to generate $I$ by $\mu(I)$. Recall that a prime ideal $\fp$ is weakly associated to $M$ if $\fp$
is minimal over $(0 :_A m)$ for some $m\in M$. We denote the set of
weakly associated primes of $M$ by $\wAss_AM$.
 \begin{enumerate}
\item[$\mathrm{(i)}$](see \cite[Definition 3.1(iv)]{AT}) $A$ is called
Cohen-Macaulay in the sense of ideals if  $\Ht(\fa)=\pgrade_{A}(\fa, A)$ for all ideal $\fa$ of $A$.
\item[$\mathrm{(ii)}$] (see \cite[Definition 1 and 2]{Ha})  Let $\fa$ be a finitely generated
ideal of $A$ with the property $\Ht \fa\geq \mu(\fa)$. Suppose  $\fp\in \wAss_A(A/\fa)$,  for any $\fp\in\min
(\fa)$. Then $A$ is called weak
Bourbaki unmixed.
\end{enumerate}
\end{definition}

In what follows we need the notion of $\check{C}$ech cohomology.
Let $R$ be a ring and  $\underline{x} = x_{1}, \ldots,
x_\ell$  be a sequence of elements of $R$. For an $R$-module $M$,  we denote the $i$-th $\check{C}$ech cohomology of $M$ with respect to $\underline{x}$  by $H_{\underline{x}}^{i}(M)$. For more details and properties, see \cite[Proposition 2.1]{HM}.

\begin{lemma}\label{kgrad}
Let $\{M_\gamma:\gamma\in \Gamma\}$ be a  direct system of modules over a ring $A$ and $\fa$ a finitely generated ideal of
$A$. Then $$  \inf\{\pgrade_{A}(\fa,M_\gamma)|\gamma\in \Gamma\} \leq \pgrade_{A}(\fa,{\varinjlim}_{\gamma\in \Gamma}M_\gamma) .$$
\end{lemma}

\begin{proof}
Let
$\underline{x}:=x_{1}, \ldots, x_{r}$ be a generating set of $\fa$. The $\check{C}ech$ grade of
$\fa$ on $M$ is defined by $$\Cgrade_{A}(\fa,M):=\inf\{i\in
\mathbb{N}\cup\{0\}|H_{\underline{x}}^{i}(M)\neq0\}.$$  By \cite[Proposition 2.3]{AT},
$\check{C}ech$ grade coincides with the polynomial grade. We prove the desired claim via $\check{C}ech$ grade.
Note that $$H_{\underline{x}}^{i}({\varinjlim}_{\gamma\in \Gamma}M_\gamma)\cong {\varinjlim}_{\gamma\in \Gamma}H_{\underline{x}}^{i}(M_\gamma) \  \ (\ast).$$
Indeed, let $S$ be the prime subring of $A$. Set $R:= S[x_1,\ldots,x_n]$ which is Noetherian. This is well-known that $$H_{\underline{x}R}^{i}({\varinjlim}_{\gamma\in \Gamma}M_\gamma)\cong {\varinjlim}_{\gamma\in \Gamma}H_{\underline{x }R}^{i}(M_\gamma) .$$ The desired  claim now follows by the  independence theorem, see \cite[Proposition 2.1]{HM}.
Let $t = \inf\{\Cgrade_{A}(\fa,M_\gamma)|\gamma\in \Gamma\}$. Then $H_{\underline{x}}^{j}(M_\gamma)=0$ for all $j< t$. Putting this along with $(\ast)$, one see that $H_{\underline{x}}^{j}({\varinjlim}_{\gamma\in \Gamma}M_\gamma)=0$ for all $j< t$. This gives $$ \inf\{\Cgrade_{A}(\fa,M_\gamma)|\gamma\in \Gamma\} \leq \pgrade_{A}(\fa,{\varinjlim}_{\gamma\in \Gamma}M_\gamma) , $$as  claimed.
\end{proof}

We will use the following result several times in this paper.

\begin{lemma}\label{pro} Let $\fa$ be an ideal of a ring $A$ and $M$ an
$A$-module. The following holds.
\begin{enumerate}
\item[$\mathrm{(i)}$] If $M$ is finitely generated, then $\pgrade_A(\fa,M)\leq\Ht_{M}(\fa).$
\item[$\mathrm{(ii)}$] Let $f :A\to B$ be a  flat ring homomorphism.
Then$$\pgrade_{A}(\fa,M)\leq \pgrade_{B}(\fa B,M\otimes_A B).$$
\item [$\mathrm{(iii)}$] Let $f :A\to B$ be a  ring homomorphism. Let  $\fa$ be an ideal of $A$ and $M$ a $B$-module. Then
$\pgrade_{A}(\fa,M)= \pgrade_{B}(\fa B,M).$
\item[$\mathrm{(iv)}$] If $\fb$ is an ideal of $A$ containing $\fa$, then  $\pgrade_A(\fa,M)\leq\pgrade_A(\fb,M).$
\end{enumerate}
\end{lemma}

\begin{proof}
See e.g. \cite[Section 2]{AT} and \cite[Lemma 3.2]{AT}.
\end{proof}

\begin{proposition}\label{cm}
Let $\{A_\gamma:\gamma\in \Gamma\}$ be a  direct system of Noetherian Cohen-Macaulay   rings and $\fp$ a prime ideal of $A:={\varinjlim}_{\gamma\in \Gamma} A_\gamma$. Set $\fp_{\gamma} = \fp \cap A_\gamma$ for all $\gamma \in \Gamma$. Assume that one of the following  holds:
\begin{enumerate}
\item[$\mathrm{(i)}$] $\cgrade_{A_\gamma}(\fp_\gamma,A_\gamma) = \cgrade_{A_\gamma}(\fp_\delta A_\gamma,A_\gamma)$ for all $\delta \leq \gamma $.
\item[$\mathrm{(ii)}$] $\cgrade_{A_\delta}(\fp_\delta,A_\delta) \leq  \cgrade_{A_\gamma}(\fp_\delta A_\gamma,A_\gamma)$ for all $\delta \leq \gamma $.
\end{enumerate}
  Then $\pgrade_{A}(\fp,A) = \Ht_{A}(\fp)$.
\end{proposition}

\begin{proof}
First recall that, since the rings in the direct system are Noetherian, the classical grade coincides with polynomial grade.

Let $n \leq \Ht_{A}(\fp)$. Look at the  following chain of prime ideals of $A$
$$\fp_0 \subset \fp_1 \subset \ldots \subset \fp_n = \fp.$$
For each $1 \leq i \leq n$, pick $x_i \in \fp_i \setminus \fp_{i-1}$. By choosing  $\gamma ^{\prime} \in \Gamma$  sufficiently large, we can assume that  $x_i\in A_\delta $ for all $i$ and $\delta \geq \gamma^{\prime}$. We use this to deduce that $n \leq \Ht_{A_\delta}(\fp_\delta)$ for all  $\delta \geq \gamma^{\prime}$. Consider the set $\Gamma^{\prime} := \{ \gamma \in \Gamma : \gamma^{\prime} \leq \gamma\}$. Obviously,  $A={\varinjlim}_{\gamma\in \Gamma^{\prime}} A_\gamma$.
Combine Lemma \ref{pro} (iii), (iv)  and Lemma \ref{kgrad} to conclude that
\[\begin{array}{ll}
\pgrade_A(\fp,A)&\geq \pgrade_A(\fp_{\gamma^{\prime}}A,A)\\
&= \pgrade_{A_{\gamma^{\prime}}}(\fp_{\gamma^{\prime}},A)\\
&= \pgrade_{A_{\gamma^{\prime}}}(\fp_{\gamma^{\prime}},{\varinjlim}_{\gamma\in \Gamma^{\prime}} A_\gamma) \\
&\geq \inf\{\pgrade_{A_{\gamma^{\prime}}}(\fp_{\gamma^{\prime}},A_\gamma)|\gamma\in \Gamma^{\prime}\}\\
&= \inf\{\pgrade_{A_\gamma}(\fp_{\gamma^{\prime}}A_\gamma,A_\gamma)|\gamma\in \Gamma^{\prime}\}.  \
\end{array}\]Now we prove the Proposition.
\begin{enumerate}
\item[$\mathrm{(i)}$] One has $$\pgrade_A(\fp,A) \geq \inf\{\pgrade_{A_\gamma}(\fp_{\gamma},A_\gamma)|\gamma\in \Gamma^{\prime}\}.$$
Take $\lambda\in \Gamma^{\prime}$ be such that $$\inf\{\pgrade_{A_\gamma}(\fp_{\gamma},A_\gamma)|\gamma\in \Gamma^{\prime}\} =\pgrade_{A_{\lambda}}(\fp_{\lambda},A_{\lambda})\  .$$Then,
\[\begin{array}{ll}
\pgrade_A(\fp,A)&\geq \pgrade_{A_{\lambda}}(\fp_{\lambda},A_{\lambda})   \   \\
&= \Ht_{A_{\lambda}}(\fp_{\lambda})\\
&\geq n.
\end{array}\]
\item[$\mathrm{(ii)}$]Take $\lambda\in \Gamma^{\prime}$ be such that $$\inf\{\pgrade_{A_\gamma}(\fp_{\gamma'}A_\gamma,A_\gamma)|\gamma\in \Gamma^{\prime}\} =\pgrade_{A_{\lambda}}(\fp_{\gamma^{\prime}}A_{\lambda},A_{\lambda})\  .$$Observe that
\[\begin{array}{ll}
\pgrade_A(\fp,A)&\geq  \pgrade_{A_{\lambda}}(\fp_{\gamma^{\prime}}A_{\lambda},A_{\lambda})\ \\
&\geq \pgrade_{A_{\gamma^{\prime}}}(\fp_{\gamma^{\prime}},A_{\gamma^{\prime}}) \\
&= \Ht_{A_{\gamma^{\prime}}}(\fp_{\gamma^{\prime}})\\
&\geq n.
\end{array}\]
\end{enumerate}
Thus, in both cases,
$\pgrade_{A}(\fp,A) \geq \Ht_{A}(\fp)$.
The reverse inequality is always true by Lemma \ref{pro}(i).\end{proof}

 We give a direct system of Noetherian regular rings such that its direct limit is not
Cohen-Macaulay in the sense of ideals.

\begin{example}\label{coexam}The grade type assumption on ideals in the previous result is really needed.
Indeed, in view of Example \ref{ex}, it is enough to give a valuation domain $(V,\fm)$ containing a field $k$ of characteristic zero and of dimension greater than one.
Such a thing exists, see  \cite[Example 8(ii)]{V}. By Example \ref{ex}, $V$ is a direct limit of Noetherian regular rings. Let $\underline{x}$ be a finite sequence  of nonzero
elements of $\fm$. Since $V$ is a  valuation domain, there is an
element $r$ such that $rV=(\underline{x})V$, e.g.,
$\pgrade_V(\underline{x}V,V)=1$. This enables us to formulate that $\pgrade_{V}(\fm, V)=1$. For more details, see \cite[Proposition 3.12]{AT}. Therefore,
$\Ht(\fm)\neq\pgrade_{V}(\fm, V)$.
\end{example}

\begin{corollary}\label{drnoeth}
Let $\{(A_\gamma,\fm_\gamma):\gamma\in \Gamma\}$  be a   direct system of Noetherian local Cohen-Macaulay  rings such that
$\fm_\gamma A_\delta=\fm_\delta$  for all $\gamma<\delta$. Then
${\varinjlim}_{\gamma\in \Gamma} A_\gamma$ is a  Noetherian Cohen-Macaulay ring.
\end{corollary}

\begin{proof} By the main result of \cite{O}, ${\varinjlim}_{\gamma\in \Gamma} A_\gamma$  is Noetherian.
We use Proposition \ref{cm}$(i)$ to conclude that ${\varinjlim}_{\gamma\in \Gamma} A_\gamma$  is  Cohen-Macaulay.
\end{proof}

We say a direct system $\{A_\gamma:\gamma\in \Gamma\}$ is a flat (pure) direct system, if  $A_\gamma\longrightarrow A_\delta$ is flat (pure) for all $\gamma,\delta\in\Gamma$ with $\gamma\leq\delta$.

\begin{corollary}\label{main2}
Let $\{A_\gamma:\gamma\in \Gamma\}$ be a  flat direct system of Noetherian Cohen-Macaulay  rings.
The following assertions hold.
 \begin{enumerate}
\item[$\mathrm{(i)}$]  ${\varinjlim}_{\gamma\in \Gamma} A_\gamma$ is Cohen-Macaulay in the sense of ideals.
\item[$\mathrm{(ii)}$] ${\varinjlim}_{\gamma\in \Gamma} A_\gamma$ is weak
Bourbaki unmixed.
\end{enumerate}
\end{corollary}

\begin{proof}
$\mathrm{(i)}$ Let $A:={\varinjlim}_{\gamma\in \Gamma} A_\gamma$ and $ \fp \in \Spec(A)$. Set $\fp_\gamma = \fp \cap A_\gamma$.  The homomorphism $A_\delta \to A_\gamma$ has the going down property for all $\delta, \gamma \in \Gamma$ with $\delta \leq \gamma$. Thus $\Ht_{A_\delta}(\fp_\delta) \leq \Ht_{A_\gamma}(\fp_\delta A_\gamma)$ for all  $\delta, \gamma \in \Gamma$ with $\delta \leq \gamma$. By Proposition  \ref{cm}$(ii)$, $\Ht_A(\fp) = \pgrade_A(\fp,A)$. The claim follows from \cite[Theorem 3.3]{AT}.

$\mathrm{(ii)}$ By \cite[Theorem 2.3.3]{G}, $A$ is a coherent ring.  Recall from \cite[Theorem 3.10]{AT} that Cohen-Macaulayness in the sense of ideals implies weak bourbaki unmixedness when the ring is coherent. Thus, part $(i)$ yields the claim.
\end{proof}

\begin{lemma}\label{flat}The following holds.
\begin{enumerate}
\item[$\mathrm{(i)}$] Let $\{(R_\gamma, f_{\gamma \delta})| \gamma,\delta \in \Gamma\}$ and $\{(S_\gamma, g_{\gamma \delta})| \gamma, \delta \in \Gamma\}$  be two direct systems of rings and $\{\varphi _\gamma : R_\gamma \to S_\gamma | \gamma \in \Gamma\} $ a morphism between them.
    Assume  for each $\gamma \in \Gamma$, $ \varphi _\gamma$ is pure (resp. flat, satisfying lying over property). Set
    $\varphi := \varinjlim_{\gamma \in \Gamma} \varphi_{\gamma}$. Then $\varphi:  \varinjlim_{\gamma \in \Gamma}{R}_{\gamma} \to  \varinjlim_{\gamma \in \Gamma}{S}_{\gamma} $ is pure (resp. flat, satisfying lying over property).
\item[(ii)]
Let $\{R_\gamma:\gamma\in \Gamma\}$  be a pure (flat, satisfying lying over property) direct system of rings. Then
 $R_\gamma \to {\varinjlim}_{\gamma\in \Gamma} R_\gamma$ is pure (flat, satisfying lying over property) for all $\gamma\in \Gamma$.
 \end{enumerate}
\end{lemma}

\begin{proof}
$\mathrm{(i)}$:
First, we prove  the purity.
Set $R := \varinjlim_{\gamma \in \Gamma}{R}_{\gamma}$ and $S := \varinjlim_{\gamma \in \Gamma}{S}_{\gamma} $.
Let $M$ be an $R$-module. Then $\pi_ \gamma: = Id_M \otimes \phi_\gamma : M \otimes _{R_\gamma} R_\gamma \to M \otimes _{R_\gamma} S_\gamma$ is an embedding. Also, $\{M \otimes _{R_\gamma} R_\gamma | \gamma \in \Gamma\}$ and $\{M \otimes _{R_\gamma} S_\gamma | \gamma \in \Gamma\}$ are direct systems of Abelian groups and $\pi := \{\pi_\gamma | \gamma \in \Gamma\}$ is a morphism of these systems. This means that $\pi: \varinjlim_{\gamma \in \Gamma} M \otimes _{R_\gamma} R_\gamma \to \varinjlim_{\gamma \in \Gamma} M \otimes _{R_\gamma} S_\gamma$ is an injection. By \cite[Chapter vi, Exercise 17]{CE}, $ M \otimes _{R} S \cong \varinjlim_{\gamma \in \Gamma} M \otimes _{R_\gamma} S_\gamma $. Hence  $M \otimes _{R} R \to  M \otimes _{R} S$ is one to one. Consequently, $R \to S$ is pure.

The proof in the  flat case is similar. For lying over see \cite{DF}.

$\mathrm{(ii)}$: Follows by part $(i)$.
\end{proof}

\begin{lemma}\label{to}
Let $\{(A_\gamma, m_\gamma):\gamma\in \Gamma\}$ be a pure direct system of Noetherian  local rings   with   constituent local homomorphism of rings $A_\gamma \to A_\delta$ for all $(\gamma , \delta) \in \Gamma \times \Gamma$ with $\gamma \leq \delta$. If the ring  $A:= {\varinjlim}_{\gamma\in \Gamma}A_\gamma$ is Noetherian, then  there exists $\lambda \in \Gamma$ such that
 \begin{enumerate}
\item[$\mathrm{(i)}$] $\dim(A_\gamma) = \dim(A)$ for all $\lambda \leq \gamma$.
\item[$\mathrm{(ii)}$]$\fm_\gamma A_\delta = \fm_\delta$ for all $\lambda \leq \gamma \leq \delta$.
\end{enumerate}
\end{lemma}
\begin{proof}
$\mathrm{(i)}$: It is easy to see  $A$ is  local with the  maximal ideal $\fm := {\varinjlim}_{\gamma\in \Gamma}{\fm_\gamma}$ and $\fm \cap A_\gamma = \fm_\gamma$ for all $\gamma \in \Gamma$.
  Suppose $\fm$ is generated by $y_1,\ldots,y_t$. Then there exists $\lambda \in \Gamma$ such that $y_i \in A_\gamma$ for all $i$ and $\lambda \leq \gamma$. Also, $\fm = \fm_\gamma A$.  In view of \cite[Theorem 15.1]{Mat}, $$\dim A \leq \dim A_\gamma + \dim A/{\fm_\gamma A}.$$
  Therefore $\dim A \leq \dim A_\gamma$. Recall from Lemma \ref{flat}(ii) that $A_\gamma \to A$ is pure. At this point, we appeal  \cite[Remark 4 and Corollary 5]{B} to see $\dim A_\gamma \leq \dim A$. This completes the proof.

$\mathrm{(ii)}$: In the light of purity,$$\fm_\delta =(\fm_\delta A)\cap A_\delta =(\fm_\gamma A)\cap A_\delta = (\fm_\gamma A_\delta)A\cap A_\delta = \fm_\gamma A_\delta,$$
    for all $\lambda \leq \gamma \leq \delta$. Note that $IB\cap A=I$ for any ideal $I$ of a ring $A$ and any pure extension $B$ of $A$.
\end{proof}

\begin{corollary}\label{cor2}
Let $\{(A_\gamma, \fm_\gamma):\gamma\in \Gamma\}$ be a pure  direct system of Noetherian Gorenstein local rings   with constituent local homomorphisms of rings $A_\gamma \to A_\delta$ for all $(\gamma , \delta) \in \Gamma \times \Gamma$ with $\gamma \leq \delta$. If the ring $A:= {\varinjlim}_{\gamma\in \Gamma}A_\gamma$ is  Noetherian, then $A$ is Gorenstein.
\end{corollary}

\begin{proof}
Recall from  Lemma \ref{to} that $A$ is local with the maximal ideal $\fm := {\varinjlim}_{\gamma\in \Gamma}{\fm_\gamma}$ and $\fm \cap A_\gamma = \fm_\gamma$ for all $\gamma \in \Gamma$. Also, there exists $\lambda \in \Gamma$ such that
 $\dim A_\gamma = \dim A_\delta = \dim A $ and $\fm_\gamma A_\delta = \fm_\delta$  for all $\lambda \leq \gamma \leq \delta$. It deduces from Corollary  \ref{drnoeth} that $A$ is  Cohen-Macaulay. Now we show that every parameter  ideal of $A$ is irreducible and consequently $A$ is Gorenstein.

   Let $\underline{x}$ be a  system of parameters of  $A$ such that $\underline{x}A = I \cap J$ where $I$ and $J$ are ideals of $A$. We choose $\delta \geq \lambda$, such that $A_\delta$ contains $\underline{x}$ and generators of $I$ and $J$. By Lemma \ref{flat}(ii) the ring homomorphism $A_\delta \to A$ is pure. In view of purity $(\underline{x}) A_\delta = (\underline{x}) A \cap A_\delta$.  Hence $ (\underline{x}) A_\delta = (I \cap A_\delta) \cap (J \cap A_\delta) $. Since generators of $I$ (resp. $J$)  belong to $I \cap A_\delta$ (resp. $J \cap A_\delta$), we have $(I \cap A_\delta)A = I$ (resp.$(J \cap A_\delta)A = J$). Thus it is enough to show that $(\underline{x}) A_\delta$ is irreducible. In order to prove it, we claim that $(\underline{x}) A_\delta$ is a parameter ideal in $A_\delta$. Indeed,
clearly,
$$\rad(\underline{x}A_{\delta})= \rad (\underline{x}A \cap A_\delta) = \rad (\underline{x}A) \cap A_\delta = \fm \cap A_\delta = \fm_\delta.$$
Also  $ \dim(A_\delta) = \dim(A)$.
These  say that
$\underline{x}$ is a  system of parameters of $A_\delta$.
\end{proof}

\begin{remark}\label{rem}
 Let $\{A_\gamma:\gamma\in \Gamma\}$ be a  direct system of Noetherian regular  (resp. complete intersection) rings. If
$A:={\varinjlim}_{\gamma\in \Gamma} A_\gamma$ is Noetherian, then
$A$ is regular  (resp. complete intersection). The regular case is in  \cite[Lemma 1.4]{S}.
For the complete intersection  case, we show $A_\fp$ is complete intersection for all $\fp\in \Spec(A)$.
 Set $\fp_\gamma:=\fp \cap A_\gamma$.  By using  $\{A_\gamma:\gamma\in \Gamma\}$, we can construct a direct system of local rings  $\{(A_\gamma)_{\fp_\gamma}:\gamma\in \Gamma\}$. It is easy to show that $A_{\fp}\cong{\varinjlim}_{\gamma\in \Gamma} (A_\gamma)_{\fp_\gamma}$.
 Then, it is enough to consider the case $\{(A_\gamma, \fm_\gamma):\gamma\in \Gamma\}$ is a direct system of Noetherian complete intersection local rings, with constituent local homomorphism of rings $A_\gamma \to A_\delta$ for all $(\gamma , \delta) \in \Gamma \times \Gamma$ with $\gamma \leq \delta$. Then $\fm :={\varinjlim}_{\gamma\in \Gamma}m_\gamma$ is the unique maximal ideal of $A ={\varinjlim}_{\gamma\in \Gamma}A_\gamma$. Thus $ \fm_\gamma = \fm \cap A_\gamma$ for all $\gamma \in \Gamma$.
Hence, $A/ \fm:={\varinjlim}_{\gamma\in \Gamma}A_\gamma /\fm_\gamma$.
 In the light of \cite[1.4.8]{MR}, $$H_n({\varinjlim}_{\gamma\in \Gamma}A_\gamma,{\varinjlim}A_\gamma/\fm_\gamma, {\varinjlim}_{\gamma\in \Gamma}A_\gamma /\fm_\gamma)\cong
{\varinjlim}_{\gamma\in \Gamma}H_n(A_\gamma,A_\gamma /\fm_\gamma, A_\gamma /\fm_\gamma).$$
It remains to apply \cite[4.6 Corollary]{Av}.
\end{remark}

\begin{corollary}\label{cor1}
Let $\{(A_\gamma, m_\gamma):\gamma\in \Gamma\}$ be a pure direct system of Noetherian regular local rings   with   constituent local homomorphism of rings $A_\gamma \to A_\delta$ for all $(\gamma , \delta) \in \Gamma \times \Gamma$ with $\gamma \leq \delta$. If $A:= {\varinjlim}_{\gamma\in \Gamma}A_\gamma$ is Noetherian,   then there exists $\lambda \in \Gamma$ such that the direct system $\{A_\gamma: \lambda \leq \gamma ,\gamma \in \Gamma \}$ is flat.
\end{corollary}

\begin{proof} By Lemma \ref{to}, there exists $\lambda \in \Gamma$ such that  $\dim A_\gamma = \dim A_\delta$ and $\fm_\gamma A_\delta = \fm_\delta$ for all $(\gamma , \delta) \in \Gamma \times \Gamma$ with $\lambda \leq \gamma \leq \delta$. Then     $$\dim A_\delta = \dim A_{\gamma} +\dim A_\delta/ \fm_{\gamma} A_\delta.$$
 By applying \cite[Theorem 23.1]{Mat}, we observe  that $A_{\gamma}\to A_\delta $ is  flat.
\end{proof}

\section{Cohen-Macaulayness  in the sense of Hamilton-Marley}

In this Section we study  the behavior of Cohen-Macaulayness  in the sense of Hamilton-Marley under taking direct limits.
Let us recall some notions.
Let $R$ be a ring, $M$ an $R$-module and  $\underline{x} = x_{1}, \ldots,
x_\ell$ be a system of elements of $R$. By $\mathbb{K}_{\bullet}(\underline{x})$  we mean the Koszul complex of $R$ with respect to $\underline{x}$.  For $m \geq n$ there exists a
chain map $$\varphi^{m} _{n} (\underline{x})
:\mathbb{K}_{\bullet}(\underline{x}^{m})\lo
\mathbb{K}_{\bullet}(\underline{x}^{n}),$$ which is induced by
multiplication of $(\prod x_{i})^{m-n}$. Recall from \cite[Definition 2.3]{Sch} that $\underline{x}$ is
\textit{weakly proregular}, if for each $n>0$ there exists an $m \geq n$ such
that the maps $$H_{i}(\varphi^{m} _{n} (\underline{x})) :H_{i}
(\mathbb{K}_{\bullet}(\underline{x}^{m}))\lo H_{i}
(\mathbb{K}_{\bullet}(\underline{x}^{n}))$$ are zero for all $i \geq
1$.

\begin{definition}\label{defhm}
Adopt the above notation.  \begin{enumerate}
\item[$\mathrm{(i)}$] Following
\cite[Definition 3.1]{HM},
$\underline{x}$ is called a \textit{parameter sequence} on
$R$, if
(1) $\underline{x}$ is a weak proregular sequence; (2)
$(\underline{x})R \neq R$, and (3) $H^{\ell}_{\underline{x}}
(R)_{\fp} \neq 0$ for all $\fp \in
\V(\underline{x}R)$. Also, $\underline{x}$ is called a strong
parameter sequence on $R$ if $x_{1},\ldots, x_{i}$ is a parameter
sequence on $R$ for all $1\leq i \leq \ell$.
\item[$\mathrm{(ii)}$]In view of \cite[Definition 4.1]{HM}, $R$ is called
Cohen-Macaulay in the sense of Hamilton-Marley, if each strong
parameter sequence on $R$ is a regular sequence on $R$.
\end{enumerate}
\end{definition}

\begin{proposition}\label{purity}
Let $\{A_\gamma:\gamma\in \Gamma\}$ be a direct system of  Noetherian Cohen-Macaulay rings with maps satisfying lying over (e.g. pure or integral extension). Then $A= {\varinjlim}_{\gamma\in \Gamma} A_\gamma$ is Cohen-Macaulay  in the sense of Hamilton-Marley.

\end{proposition}

\begin{proof}
Let $\textbf{x}:=x_1,\ldots,x_n $ be a  parameter sequence on $A$. There exists $\gamma \in \Gamma$ such that $\textbf{x}\in A_\delta $ for all $\delta \geq \gamma$.
Let  $(\gamma , \delta) \in \Gamma \times \Gamma$ with $\lambda \leq \gamma \leq \delta$ and $\fp \in \Var(\textbf{x}A_\delta  )$. In view of  Lemma \ref{flat}(ii), we know that the lying over property  holds for $A_\delta \to A$. There is a prime ideal $\fq \in \Spec (A)$ such that $\fq \cap A_\delta   = \fp $. Then,
 \[\begin{array}{ll}
H_{\textbf{x}}^{n}(A)_{\fq}&\cong H_{\textbf{x}A_\fq}^{n}(A_\fq)\\
&\cong H_{\textbf{x}(A_\delta  )_\fp}^{n}(A_\fq)\\
&\cong H_{\textbf{x}(A_\delta)_\fp}^{n}((A_\delta)_\fp)\otimes_{(A_\delta )_\fp} A_{\fq}. \\
\end{array}\]
It yields that $H_{\textbf{x}(A_\delta  )_\fp}^{n}((A_\delta  )_\fp) \neq 0$. Recall that any finite sequence of elements in a   Noetherian ring is weak proregular, i.e.,  $\textbf{x}$ is a  parameter sequence over $A_\delta $.

Let $\textbf{y}:=y_1,\ldots,y_m $ be a strong  parameter sequence on $A$. Now, one can easily find $\gamma \in \Gamma$ such that $ \textbf{y}$ is a strong parameter sequence on $A_\delta$ for all $\delta \geq \gamma$. Recall that a   Noetherian  Cohen-Macaulay ring is Cohen-Macaulay  in the sense of Hamilton-Marley.
Thus, by Cohen-Macaulayness of $A_\delta  $,  $\textbf{y}$ is a regular sequence on $A_\delta  $  for all $\delta \geq \gamma$. An easy direct limit argument implies that $\textbf{y}$ is a regular sequence on $A$.
\end{proof}

\begin{definition}\label{gen}
By $\mathbb{Z}^\infty$, we mean $\bigcup_{s\in \mathbb{N}} \mathbb{Z}^s$.
Let  $ M , N\subseteq \mathbb{Z}^\infty$ be   monoids and $M \subseteq N$. \begin{enumerate}
\item[$\mathrm{(i)}$] $M$ is called full in $N$, if  $\alpha, \alpha^\prime \in M$ and  $\alpha -\alpha^\prime \in  N $
, then $\alpha -\alpha^\prime \in M$.
\item[$\mathrm{(ii)}$] $M$ is called locally finitely generated, if $M\cap \mathbb{Z}^s$ is finitely generated for all $s\in \mathbb{N}$. \item[$\mathrm{(iii)}$]
    Recall that $M$ is said to be \emph{normal} if, whenever
$n( m-m^\prime) \in C$ for some positive integer $n$
 and $m,m^\prime \in M$, then $ m-m^\prime \in M$.
\end{enumerate}
\end{definition}

\begin{corollary}\label{semi2}
Let $C\subseteq \mathbb{Z}^\infty$ be a  normal locally finitely generated monoid. Then $k[C]$  is Cohen-Macaulay in the sense of Hamilton-Marley.
\end{corollary}

\begin{proof}
Set $C_n:= C\cap \mathbb{Z}^n$ for every $n\in \mathbb{N}$. Then $C_n$ is a normal monoid of $\mathbb{Z}^n$ and $C_n\subseteq C_{n+1}$. By \cite[Theorem 6.3.5]{BH}, $k[C_n]$ is a Noetherian Cohen-Macaulay  ring. Note that $C_n$ is full in $C_{n+1}$. It turns out that
$k[C_n] \hookrightarrow k[C_{n+1}]$ is pure, see \cite[Exercise 6.1.10(c)]{BH}.
Clearly, $k[C] = {\varinjlim}_{n\in \mathbb{N}}k[C_n]$. We can now use Proposition $\ref{purity}$ to conclude that $k[C]$  is Cohen-Macaulay in the sense of Hamilton-Marley.
\end{proof}

\begin{example}\label{br2}
Let $k$  be a field, $A:=k[X_1,\ldots]:=\bigcup_{i=1}^{\infty}R[X_1,\ldots,X_i]$ and $f:=X_{i_1}^{j_1}\ldots X_{i_{\ell}}^{j_{\ell}}\in A$. The degree of $f$ is defined by $d(f):=\sum_{k=1}^{\ell} j_k$. Let $d$ be a positive integer. We call the $k$-algebra  $A^{(d)}$ generated by all monomials of degree $d$, the $d$-th veronese subring of $A$. Then $A^{(d)}$ is Cohen-Macaulay in the sense of Hamilton-Marley. Indeed, look at
 $$C:=\{(n_i) \in  \mathbb{Z}^\infty :n_i\geq0 \textit{ and  } d \mid \sum n_i\}. $$ It remains to recall that
$A^{(d)}:=k[C]$ and $C$ is  normal and locally finitely generated.  We leave the details to the reader.
\end{example}

Our next aim is to construct an example to show that the assumptions of Proposition \ref{purity} are essential, see Theorem \ref{coexamhm2} for its application.
Let $\fa$ be an ideal of a ring $A$. By $\Var_{A}(\fa)$, we mean $\{\fp\in\Spec (A):\fa\subseteq \fp\}$.
\begin{lemma}\label{lem1}
Let $k$  be a field and  $A:=k+xk[x,y]$. Then
$\Var_{A}(xA) = \{ xk[x,y]\}$.
\end{lemma}

\begin{proof}
Let $\fp \in \Spec(A)$ be such that $xA \subseteq \fp$.    Then $(xy^n)^2 \in \fp$ and so $xy^n \in \fp$ for all $n \in \mathbb{N}_0$. This implies that $xk[x,y] \subseteq \fp$. It is not difficult to observe that $xk[x,y]$ is a maximal ideal of $A$. Therefore $xk[x,y] = \fp$, as claimed.
\end{proof}

\begin{example}\label{coexamhm1}
Let $k$  be a field. For each $n$, set $H_n := \{(a,b) \in \mathbb{N}_0^2| 0 \leq b/a \leq n\} \cup \{(0,0)\}$ and $H := \{(a,b) \in \mathbb{N}_0^2| 0 \leq b/a < \infty\} \cup \{(0,0)\}$. Then the following holds:
\begin{enumerate}
\item[$\mathrm{(i)}$] $H_n$ is a normal submonoid of $\mathbb{Z}^2$.
\item[$\mathrm{(ii)}$] $H_n$ is finitely generated, and so $k[H_n]$ is  Noetherian and Cohen-Macaulay.
\item[$\mathrm{(iii)}$] For each $n \in \mathbb{N}_{0}$, set $A_n := k[H_n]$ and $A :=k + xk[x,y]$. Then $$A = k[H] = \bigcup_{n \in \mathbb{N}_0} A_n.$$
\item[$ \mathrm{(iv)}$] Let $\pi_n : A_n/xyA_n \to  A_{n+1}/xyA_{n+1}$ be the natural homomorphism of rings  for all $n \in \mathbb{N}_0$. Then $\{A_n /xyA_n, \pi_n\}_{n \in \mathbb{N}_0}$ is a direct system of Noetherian Cohen-Macaulay rings such that $$A/xyA \cong {\varinjlim}_{n\in \mathbb{N}_0} A_n/xyA_n$$ and $A/xyA$ is not Cohen-Macaulay in the sense of Hamilton-Marley.
\end{enumerate}
\end{example}

\begin{proof}
$\mathrm{(i)}$: First, we remark  trivially that, $H_n$ is a monoid.
Let $(a,b),(c,d)\in H_n$ and suppose there is $k \in \mathbb{N}$ such that $k((a,b)-(c,d))\in  H_n$.
 Clearly, $$0 \leq (b-d)/(a-c)=k(b-d)/k(a-c)\leq n,$$and so $(a,b)-(c,d)\in  H_n$. This completes the proof.

$\mathrm{(ii)}$: For each $n$, let  $C_n\subseteq H_n$ be the submonoid generated by $\{(1,0),(1,1), \ldots,(1,n)\}$. We prove $C_n= H_n$ by induction on $n$. When $n=0$ there are nothing to prove. Now suppose, inductively, that $ n > 0$ and the result has been proved for $n-1$.  Then $H_{n-1} \subseteq C_n$. Take $(a,b)\in H_{n}\setminus H_{n-1}$.
This means that $ n-1 < b/a \leq n$. It implies  that $b = an - i$ where $0 \leq i \leq a-1$. In the light of $$(a,b) = (a,an-i) = i(1,n-1)+ (a-i)(1,n)$$ we conclude  that $(a,b) \in C_n$. So $H_n \subseteq C_n$. By \cite[Theorem 6.3.5]{BH}, $K[H_n]$ is a Noetherian  Cohen-Macaulay ring.

$\mathrm{(iii)}$: Recall from  Example \ref{coexamhm} that $ k[H] = A$. Also, $k[H] = \bigcup_{n \in {\mathbb{N}_0}} A_n $, because $H = \bigcup_{n \in {\mathbb{N}_0}}H_n$.

$\mathrm{(iv)}$:
Clearly $\bar{A}_n := A_n /xyA_n$ is Cohen-Macaulay and $\bar{A} := A/xyA = {\varinjlim}_{n\in \mathbb{N}}\bar{A}_n$.
 Denote the natural image of $a \in A$ in $\bar{A}$ by $\bar{a}$. Keep  in mind that $\bar{y}\notin \bar{A}$.  This force that $\bar{xy}^2 \neq 0$ but $\bar{x} \bar{xy}^2 = 0$. Thus $\bar{x}$ is not regular.
Having this, in order to show $\bar{A}$ is not Cohen-Macaulay in the sense of Hamilton-Marley, it is enough to show $\bar{x}$ is parameter in $\bar{A}$.
 Clearly, $\bar{x}\bar{A} \neq \bar{A}$. Suppose $H_{\bar{\textbf{x}}\bar{A}}^{1}(\bar{A})_{\fq} = 0$ for some $\fq \in \Var_{\bar{A}}(\bar{x}\bar{A})$ and  search for a contradiction. We have
 $\Supp_{\bar{A}}(H_{\bar{\textbf{x}}\bar{A}}^{1}(\bar{A}))\subseteq \Var_{\bar{A}}(x\bar{A})$.
 So, in view of Lemma \ref{lem1},
$0 =H_{\bar{\textbf{x}}\bar{A}}^{1}(\bar{A})= \bar{A}_{\bar{x}}/{\bar{A}}$. Thus, there exists $f \in A$ such that  $1/{\bar{x}} = \bar{f}/1$ in $\bar{A}_{\bar{x}}$. Consequently $(1 - fx)x^i = xyg$ for some $g\in A$ and $i \in \mathbb{N}_0$. One can use this easily to find a contradiction  (note that $A$ is not quasi-local). Therefore, $H_{\bar{\textbf{x}}\bar{A}}^{1}(\bar{A})_{\fq} \neq 0$ for all $\fq \in \Var_{\bar{A}}(\bar{x}\bar{A})$. To complete the proof, we need to prove $\bar{x}$ is  weakly proregular. In the light of \cite[Page 346]{HM}, it remains to show $$(xyA :_A x) = xyk[x,y] = (xyA:_A x^2).$$ We do this in the lines through with \cite[Example 4.9]{HM}.    Clearly, $xyk[x,y] \subseteq (xyA :_A x^2)$. Reversely, take $a \in (xyA :_A x^2)$. Then $ax^2 = xy a^{\prime}$ for some $a^{\prime} \in A$, i.e., $ax = y(c + bx)=yc + bxy$ where $c\in k$ and $b\in k[x,y]$  and conclude that $yc \in A$. Therefore $c=0$ and so $a = by $. Let $a = d + xh$ where $d \in k$ and $h \in k[x,y]$. Then $by = d + xh$. Thus $d=0$ and so $a = xh = by$. This means $a \in xyk[x,y]$. This completes the proof of the first equality. The proof of $(xyA :_A x) = xyk[x,y]$ is similar.
\end{proof}

Our next result is Theorem \ref{coexamhm2}. To prove it we need a couple of lemmas.

\begin{lemma}\label{1}
Let $H$ be the normal monoid $  \{(a,b) \in \mathbb{N}_0^2| 0 \leq b/a < \infty\} \cup \{(0,0)\}$ and let $k$ be a  field.
Let $f$ and $g$  be in $xk[x,y]$. Then $f,g$ is not a parameter sequence in $k[H]$.
\end{lemma}

\begin{proof}
By Example \ref{coexamhm1}, we know that $A := k[H]= k + xk[x,y]$.
Suppose on the contrary that $f,g \in xk[x,y]$ is  a parameter sequence.
By \cite[Theorem 21.4]{Gi}, $\dim k[H]=\dim k[\mathbb{Z}^2]$. It follows by
\cite[Theorem 17.1]{Gi} that  $\dim k[\mathbb{Z}^2]=2$.
Now, in view of  \cite[Proposition 3.6]{HM}, we deduce that $\Ht_{A}((f,g)A) = 2$. Let $\fm := xk[x,y]$.  Hence   $ \fm \in \min_A((f,g)A)$.
 By Lemma \ref{lem1},
$$\rad_{A_{\fm}} (x) = \fm A_{\fm}=\rad _{A_{\fm}}(f,g).$$ So $H_{f,g}^2(A)_\fm = H_{x}^2(A)_\fm = 0$, a contradiction.
\end{proof}

\begin{lemma}\label{2}
Adopt the  notation of Lemma \ref{1} and let
$f\in k[H]$ be such that $f(0)\neq 0$.  If $\fp \in \Var_{k[x,y]}(f)$ is of height one, then there is $f_1\in k[H]$ such that $\fp=f_1k[x,y]$ and $f_1(0)\neq 0$.
\end{lemma}

\begin{proof}  Note that $\fp = f_1 k[x,y]$ for some $f_1 \in k[x,y]$, because $k[x,y]$ is an unique factorization domain, see \cite[Theorem 20.1]{Mat}. We show that $f_1 \in  k[H]$. Represent $f_1$ by $ a + \sum_{i=1}^{i=n} h_i$ where $a \in k$ and $h_i$ are homogeneous elements of degree $i$. Set $$I:=\{i|x \textit{ divides }h_i\}$$ and $J:=\{1, \ldots,n\}\setminus I$.
Since $f \in  f_1 k[x,y]$, we have $f = f_1 g_1$ for some  $g_1  \in k[x,y]$. Note that $g_1 =\sum_{i=1}^{i=m} g_i^{\prime} + b$ where $b \in k$ and  $g_i^ {\prime}$ are homogeneous elements of degree $i$. Set $$I':=\{i|x \textit{ divides }g_i'\}$$ and $J':=\{1, \ldots,m\}\setminus I'$.  Write $f$ as $c + xh$ for some $0 \neq c \in k$ and $h \in k[x,y]$. Hence $$ c + xh = (a + h_I + h_J)(g_{I^{\prime}}+ g_{J^{\prime}} + b),$$ where $h_I:=\sum_{i \in I} h_i$ and $h_J:=\sum_{j \in J}h_j$. The terms $g_{I^{\prime}
}, g_{J^{\prime}}$ are defined in a similar way. Note that $x$ divides $ h_J g_{J^{\prime}} + h_J b+ a g_{J^{\prime}} $. Choose $i\in J$ and $j\in{J^{\prime}}$ be such that $h_ig_j^{\prime}$ is a homogeneous term in $h_J g_{J^{\prime}}$ of maximal degree. Thus $x \mid h_ig_j^{\prime}$. So, either  $h_J =0$ or $g_{J^{\prime}}=0$. If $g_{J^{\prime}}=0$, then $x \mid h_J b$. This means that $h_J =0$. In both cases,  $h_J =0$. Therefore,  $f_1 \in k[H]$ as claimed. The fact  $f_1(0)\neq 0$ follows easily from $f(0)\neq 0$ and $f=f_1g_1$.
\end{proof}

\begin{theorem}\label{coexamhm2}
Let $H$ be the normal semigroup $  \{(a,b) \in \mathbb{N}_0^2| 0 \leq b/a < \infty\} \cup \{(0,0)\}$ and let $k$ be a  field. Then the following statements are true.
\begin{enumerate}
\item[$\mathrm{(i)}$]$k[H]$ is not  Cohen-Macaulay in the sense of ideals.
\item[$\mathrm{(ii)}$]$k[H]$ is not weak bourbaki unmixed.
\item[$\mathrm{(iii)}$]
 $k[H]$ is  Cohen-Macaulay in the sense of Hamilton-Marley.
\end{enumerate}
\end{theorem}

\begin{proof}
We use the notation of Example \ref{coexamhm1}.

$\mathrm{(i)}$: Suppose on the contrary that $A := k[H]$ is Cohen-Macaulay in the sense of ideals.
Lemma 3.5 in \cite{AT} implies that $\bar{A} :=k[H]/(xy)$ is Cohen-Macaulay in the sense of ideals.
It follows from \cite[Theorem 3.4]{AT} that $\bar{A}$ is Cohen-Macaulay in the sense of Hamilton-Marley.
But Example \ref{coexamhm1} says that $\bar{A}$ is not Cohen-Macaulay in the sense of Hamilton-Marley. This is a contradiction.

$\mathrm{(ii)}$: Set $A := k[H]$ and $\fp= (xy,xy^2,\ldots)A$. First we show $\fp\in\Spec(A)$.
The assignment $ f + \fp\mapsto f(x,0)$
defines a ring homomorphism $\pi : A/\fp \to k[x]$.
Note that any  $f \in A$ is of the form $f = c + x (\sum_{i=0}^{i=n}f_i y^i)$ where $f_i \in k[x]$ and $c\in k$.
This says  that $\pi$ is injective. Trivially,  $\pi$ is surjective. So  $\fp\in\Spec(A)$.

Keep in mind that $\Ht_A(xyA) \geq 1$.
Thus we can deduce the claim from the following three statements:

\begin{enumerate}
\item[$\mathrm{(a)}$] $\min_A(xyA) = \{\fp\}$. Indeed, we show every prime ideal containing $xy$ contains $\fp$.  Let $\fq \in \Var_A(xyA)$. If $x \in \fq$, then by Lemma \ref{lem1}, $\fq = xk[x,y]$. So $\fp\subseteq\fq$. Thus, without loss of the   generality, we can assume  that $x\notin\fq$. It follows from   $x^{m-1}(xy^m)\in\fq$ for all $m \geq 2$ that  $xy^m \in \fq$ for all $m \geq 2$. Hence $\fp \subseteq \fq$.
\item[$\mathrm{(b)}$] $xk[x,y] \in \wAss_A(A/xyA)$. Indeed, note that $\bar{x}$ is zero divisor in $\bar{A}$. Hence by \cite[Theorem 3.3.1(ii)]{G}, there exists  $\fq \in \wAss_A(A/xyA)$ such that $x  \in  \fq$. By Lemma \ref{lem1}, $\fq = xk[x,y]$. So $xk[x,y] \in \wAss_A(A/xyA)$.
\item[$\mathrm{(c)}$]  $xk[x,y] \neq \fp$. This is trivial, because $x\notin \fp$.
\end{enumerate}

$\mathrm{(iii)}$: Recall that $A := k[H]= k + xk[x,y]$. Incorporate  \cite[Proposition 3.6]{HM} with \cite[Theorem  17.1 and 21.4]{Gi} to deduce that the length of any parameter sequence
is bounded by two. Since $A$ is domain, any parameter element is regular. It remains to study  parameter sequences of length two.
Let $f,g \in A$ be a parameter sequence.   Keep in mind that permutation of a parameter sequence is a parameter sequence. In view of Lemma \ref{1}, without loss of the generality, we can assume that $f := c + xh$ where $0 \neq c \in k$ and $h \in k[x,y]$. Look at the naturel inclusion $A \hookrightarrow k[x,y]$. Clearly, $fA \subseteq (f) k[x,y]  \cap A$. Reversely, suppose $fg\in A$ for some $g\in k[x,y]$. Let $d:=g(0,0)$ and $\overline{g}:=g-d$.
 Then $\overline{g}c + x\overline{g}h + dxh + cd \in A$. This yields that $\overline{g}=xg'$ for some $g'\in k[x,y]$. Thus, $g=\overline{g}+ d \in A $. This, truly, claims that  $$fA = (f) k[x,y]  \cap A  \  \ (\ast).$$ Now we use this fact to reduce the issue in the case $(f,g)k[x,y]\subsetneqq k[x,y]$.
Indeed, if $(f,g)k[x,y] = k[x,y]$, then $g + fk[x,y]$ is unit in the ring $k[x,y]/fk[x,y]$. Note that  $(f) k[x,y]  \cap A = (f) A$. This implies that $f,g$ is a regular sequence in $A$. Thus we can assume $$(f,g)k[x,y]\subsetneqq k[x,y].$$
There are two  possibilities:
\begin{enumerate}
\item[$\mathrm{(a)}$]: Height of $(f,g)k[x,y]$ is one. Let $\fp \in \Var_{k[x,y]}((f,g)k[x,y])$ be such that $\Ht(\fp)=1$. By Lemma \ref{2}, there is $f_1 \in A$ such that
 $\fp=f_1k[x,y]$ and $ a:=f_1(0) \neq 0$. Then by the same idea as $(\ast)$,
    $\fp \cap A = f_1 A$. Denote $\fq:= f_1A$. Then, since $$\Ht_A((f,g)A)=\dim A=2,$$ $\fq$ is a minimal prime ideal of $(f,g)A$. Look at $H_{f,g}^2(A)_\fq = H_{f_1}^2(A)_\fq = 0$. This  contradiction says that $\mathrm{(a)}$ is not the case.
\item[$\mathrm{(b)}$]:  Height of $(f,g)k[x,y]$ is two. It follows from  Cohen-Macaulayness of $k[x,y]$ that $f,g$ is a regular sequence in $k[x,y]$. Since $fk[x,y] \cap A = f A$, we deduce $f,g$ is  a regular sequence in $A$.
\end{enumerate}
The proof is now  complete.
\end{proof}

\begin{remark}
Let $A=k[H]$ be as Theorem \ref{coexamhm2}.

$\mathrm{(i)}$
 By \cite[Theorem 3.10]{AT}, Cohen-Macaulayness in the sense of ideals implies weak bourbaki unmixedness when the ring is coherent.
In order to show  Theorem  \ref{coexamhm2}$\mathrm{(i)}$ is independent from Theorem \ref{coexamhm2}$\mathrm{(ii)}$,
we claim that $A$ is not coherent. Indeed, $(xyA :_A x) = xyk[x,y]$ is not finitely generated as  an ideal of $A$.
By part (6) of \cite[Theorem 2.3.2]{G}, $A$ is not coherent.

$\mathrm{(ii)}$ Recall from \cite{ADEH} that  a ring $R$ satisfies
\textmd{GPIT}, if
$\Ht(\fp)\leq n$ for each prime ideal $\fp$ of $R$ which is minimal
over an $n$-generated ideal of $R$. Any Cohen-Macaulay ring in the sense of ideals satisfies \textmd{GPIT}. The ring $A$
does not satisfy \textmd{GPIT}, because the principal ideal $xA$ is of height two.
To this end, in view of Lemma \ref{lem1}, $\min(xA)=xk[x,y]$. By the display item $\mathrm{(a)}$ in the proof of Theorem \ref{coexamhm2}(ii), $(xy,xy^2,\ldots)A$
is prime. Thus $\Ht(xA)\geq 2$. We know that $\dim A=2$. So $\Ht(xA)= 2$. This provides
a second proof for
 Theorem  \ref{coexamhm2}(i).
\end{remark}

\section{Concluding remarks and questions}

Inspired by  Theorem \ref{coexamhm2} we ask the following:

\begin{question}\label{1}Let $H$ be a normal submonoid of $\mathbb{Z}^n$.
Is $k[H]$ Cohen-Macaulay  in the sense of Hamilton-Marley?
\end{question}

Inspired by the proof  of Theorem \ref{coexamhm2} and as suggested by the referee, we ask the following question:

\begin{question}\label{2} Suppose $R$ is a graded ring such that every homogeneous (monomial)
strong parameter sequence is a regular sequence. Must $R$ be Cohen-Macaulay  in the sense of Hamilton-Marley?
\end{question}

\begin{remark}
Question \ref{1} follows by Question \ref{2}. Indeed,
it is proved in \cite{AD} that any monomial
strong parameter sequence  of $k[H]$ is a regular sequence.
Clearly, this yields the claim.
\end{remark}

\begin{acknowledgement}
We thank the anonymous referee for his/her detailed review.
\end{acknowledgement}


\end{document}